\documentclass[12pt,reqno]{amsart}

\usepackage[all]{xy}
\usepackage{url}
\usepackage{amssymb}
\usepackage{hyperref}
\usepackage[margin=1.2 in]{geometry}
\usepackage{color,soul}
\usepackage{amsthm}
\usepackage{lipsum}
\usepackage{mathtools}
\DeclarePairedDelimiter{\ceil}{\lceil}{\rceil}
\DeclareMathOperator{\Spec}{Spec}
\DeclareMathOperator{\td}{tr.deg}
\DeclareMathOperator{\rr}{rat.rk}
\theoremstyle{plain}
\newtheorem{theorem}[subsubsection]{Theorem}
\newtheorem{proposition}[subsubsection]{Proposition}
\newtheorem{lemma}[subsubsection]{Lemma}

\newtheorem{corollary}[subsubsection]{Corollary}
\newtheorem{claim}[subsubsection]{Claim}

\theoremstyle{definition}
\newtheorem{definition}[subsubsection]{Definition}
\newtheorem{example}[subsubsection]{Example}
\newtheorem{examples}[subsubsection]{Examples}

\newtheorem{remark}[subsubsection]{Remark}
\newtheorem{remarks}[subsubsection]{Remarks}
\newtheorem{setup}[subsubsection]{Setup}

\theoremstyle{definition}

\theoremstyle{plain}
\newtheorem{innercustomthm}{Theorem}
\newenvironment{customthm}[1]
  {\renewcommand\theinnercustomthm{#1}\innercustomthm}
  {\endinnercustomthm}

\theoremstyle{plain}  
\newtheorem{innercustomcor}{Corollary}
\newenvironment{customcor}[1]
  {\renewcommand\theinnercustomcor{#1}\innercustomcor}
  {\endinnercustomcor}

\numberwithin{equation}{subsubsection}
\numberwithin{equation}{subsubsection}
\normalfont
\usepackage[T1]{fontenc}{}
\usepackage{setspace}

\usepackage{tikz-cd}

\makeatletter
\def\@tocline#1#2#3#4#5#6#7{\relax
  \ifnum #1>\c@tocdepth 
  \else
    \par \addpenalty\@secpenalty\addvspace{#2}%
    \begingroup \hyphenpenalty\@M
    \@ifempty{#4}{%
      \@tempdima\csname r@tocindentvmber#1\endcsname\relax
    }{%
      \@tempdima#4\relax
    }%
    \parindent\z@ \leftskip#3\relax \advance\leftskip\@tempdima\relax
    \rightskip\@pnumwidth plus4em \parfillskip-\@pnumwidth
    #5\leavevmode\hskip-\@tempdima
      \ifcase #1
       \or\or \hskip 1em \or \hskip 2em \else \hskip 3em \fi%
      #6\nobreak\relax
    \dotfill\hbox to\@pnumwidth{\@tocpagenum{#7}}\par
    \nobreak
    \endgroup
  \fi}
\makeatother
\usepackage{hyperref}

\title{Uniform approximation of Abhyankar valuation ideals in function fields of prime characteristic}
\author{Rankeya Datta\vspace{-2ex}}
\address{Department of Mathematics, Statistics and Computer Science, University of Illinois at Chicago, Chicago, IL 60607, USA}
\email{rankeya@uic.edu}
\thanks{The author's work was supported by a department fellowship at University of Michigan and Karen Smith's NSF grant DMS $\# 1501625$. Preliminary results were obtained during the author's stay at the University of Utah, and he is grateful to Karl Schwede for the invitation.
}

\begin{document}

\maketitle
\begin{abstract}
We prove the prime characteristic analogue of a characteristic $0$ result of Ein, Lazarsfeld and Smith \cite{ELS03} on uniform approximation of valuation ideals associated to real-valued Abhyankar valuations centered on regular varieties.
\end{abstract}

Let $X$ be a variety over a perfect field $k$ of prime characteristic $p > 0$, with function field $K$. Suppose $\nu$ is a real-valued valuation of $K/k$ centered on $X$. Then for all $m \in \mathbb{R}$, we have the \emph{valuation ideals} 
$$\mathfrak{a}_m(X) \subseteq \mathcal{O}_X,$$ 
consisting of local sections $f$ such that $\nu(f) \geq m$. When $X = \text{Spec}(A)$, we denote the ideal $\{a \in A: \nu(a) \geq m\}$ of $A$ by $\mathfrak{a}_m(A)$.

The goal of this paper is to use the theory of asymptotic test ideals  in positive characteristic to prove the following uniform approximation result for Abhyankar valuation ideals established in the characteristic $0$ setting by Ein, Lazarsfeld and Smith \cite{ELS03}. 

\begin{customthm}{A}
\label{Theorem-A}
Let $X$ be a regular variety over a perfect field $k$ of prime characteristic with function field $K$. For any nontrivial, real-valued Abhyankar valuation $\nu$ of $K/k$ centered on $X$, there exists $e \geq 0$ such that, for all $m \in \mathbb R_{\geq 0}$ and $\ell \in \mathbb N$,
$$\mathfrak{a}_m(X)^{\ell} \subseteq \mathfrak{a}_{\ell m}(X) \subseteq \mathfrak{a}_{m - e}(X)^{\ell}.$$
\end{customthm}

Thus, the theorem says that the valuation ideals $\mathfrak{a}_{\ell m}$ associated to a real-valued Abhyankar valuation are uniformly approximated by powers of $\mathfrak{a}_{m}$. An Abhyankar valuation (Section \ref{Abhyankar valuations}) is a generalization of the notion of the order of vanishing along a prime divisor on a normal model. As such, these are the valuations that are the most geometrically tractable (see \cite{Spi90, ELS03, JF04, JM12, Tem13, Tei14, RS14, Blu18, DS16} for some applications). For example, a key point in our proof of Theorem \ref{Theorem-A} is that Abhyankar valuations over perfect fields are locally uniformizable \cite[Theorem 1]{Knaf05}, which  implies that  any real-valued Abhyankar valuation admits a center on a regular local ring such that the corresponding valuation ideals become monomial with respect to a suitable choice of a regular system of parameters (Proposition \ref{corollary-theorem-on-local-uniformization}). In other words, real-valued Abhyankar valuations over perfect fields of prime characteristic are \emph{quasi-monomial}, a result which in characteristic $0$ follows from resolution of singularities \cite[Proposition 2.8]{ELS03}. 

In \cite{ELS03} (see also \cite{Blu18}), Theorem \ref{Theorem-A} is proved over a ground field of characteristic $0$ using the machinery of \emph{asymptotic multiplier ideals}, first defined in \cite{ELS01} in order to prove a uniformity statement about symbolic powers of ideals on regular varieties. It has since become clear that in prime characteristic a \emph{test ideal} is an analogue of a multiplier ideal. Introduced by Hochster and Huneke in their work on tight closure \cite{HH90}, the first link between test and multiplier ideals was forged by Smith \cite{Smi00} and Hara \cite{Har01}, following which Hara and Yoshida introduced the notion of test ideals of pairs \cite{HY03}. Even with the failure of cohomological vanishing theorems in positive characteristic, test ideals of pairs were shown to satisfy many of the usual properties of multiplier ideals of pairs that make the latter such an effective tool in birational geometry \cite{HY03, HT04, Tak06} (see also Theorem \ref{basic-properties-test-ideals}).  

In this paper, we employ an asymptotic version of the test ideal of a pair to prove Theorem \ref{Theorem-A}, drawing inspiration from the asymptotic multiplier ideal techniques in \cite{ELS03}. However, instead of utilizing tight closure machinery, our approach to asymptotic test ideals is based on Schwede's dual reformulation of test ideals using \emph{$p^{-e}$-linear maps}, which are like maps inverse to the Frobenius homomorphism \cite{Sch10, Sch11} (see also \cite{Smi95, LS01}).

Asymptotic test ideals are associated with graded families of ideals (Definition \ref{graded-family-ideals}), an example of the latter being the family of valuation ideals $\mathfrak{a}_{\bullet} \coloneqq \{\mathfrak{a}_m(A)\}_{m \in \mathbb R_{\geq 0}}$. For each $m \geq 0$, one constructs the \emph{$m$-th asymptotic test ideal} $\tau_m(A, \mathfrak{a}_{\bullet})$ of the family $\mathfrak{a}_{\bullet}$, and then Theorem \ref{Theorem-A} is deduced using

\begin{customthm}{B}
\label{Theorem-B}
Let $\nu$ be a nontrivial, real-valued Abhyankar valuation of $K/k$, centered on a regular local ring $(A, \mathfrak m)$, where $A$ is essentially of finite type over the perfect field $k$ of prime characteristic with fraction field $K$. Then there exists $r \in A - \{0\}$ such that for all $m \in \mathbb{R}_{\geq 0}$,
$$r\cdot \tau_m(A, \mathfrak{a}_{\bullet}) \subseteq \mathfrak{a}_m(A).$$
In other words,  
$\bigcap_{m \in \mathbb{R}_{\geq 0}} (\mathfrak{a}_m: \tau_m(A, \mathfrak{a}_{\bullet})) \neq (0).$
\end{customthm}

Finally, as in \cite{ELS03}, Theorem \ref{Theorem-B} also gives a new proof of a prime characteristic version of Izumi's theorem for arbitrary real-valued Abhyankar valuations with a common regular center (see also the more general work of \cite{RS14}).

\begin{customcor}{C}[\textbf{Izumi's Theorem for Abhyankar valuations in prime characteristic}]
\label{Corollary-C}
Let $\nu$ and $w$ be nontrivial, real-valued Abhyankar valuations of $K/k$, centered on a regular local ring $(A, \mathfrak m)$, as in Theorem \ref{Theorem-B}. Then there exists a real number $C > 0$ such that for all $x \in A - \{0\}$,
$$\nu(x) \leq Cw(x).$$
\end{customcor}

\noindent Corollary \ref{Corollary-C} implies that the valuation topologies on $A$ induced by \emph{any} two nontrivial real-valued Abhyankar valuations are \emph{linearly} equivalent.

Hara defined and used asymptotic test ideals to give a prime characteristic proof of uniform bounds on symbolic power ideals \cite{Har05}, which is independent of Hochster and Huneke's earlier proof \cite{HH02} and similar to the multiplier ideal approach of \cite{ELS01}. This paper continues the efforts of Hara and other researchers to use test ideals to prove prime characteristic analogues of statements in characteristic 0 that were established using multiplier ideals.

\noindent{\textbf{Structure of the paper:}} Section \ref{Conventions} establishes notation, and Section \ref{section-Abhyankar-val} is a brief survey of Abhyankar valuations, including a proof of local monomialization of real-valued Abhyankar valuations over perfect fields of arbitrary characteristic (Proposition \ref{corollary-theorem-on-local-uniformization}). In Section \ref{Characteristic p preliminaries}, we collect some basic facts about the Frobenius endomorphism needed in our discussion of asymptotic test ideals. Section \ref{Test Ideals} is the technical heart of the paper, and after summarizing the construction and basic properties of test ideals of pairs, we embark on a description of asymptotic test ideals. Included are discussions on the behavior of asymptotic test ideals under essentially \'etale (Subsection \ref{test-ideals-etale-maps-subsection}) and birational maps (Subsection \ref{(Asymptotic) test ideals and birational maps}), which are needed in order to reduce the proof of Theorem \ref{Theorem-B} to the case of monomial ideals in a polynomial ring. We prove Theorem \ref{Theorem-B} in Section \ref{section-proof-theorem-B}, and in Section \ref{Consequences of Theorem B} deduce Theorem \ref{Theorem-A} and Corollary \ref{Corollary-C} using Theorem \ref{Theorem-B}.

\section{Conventions}
\label{Conventions}

All rings are commutative with unity, and ring homomorphisms preserve the multiplicative identities. We use $\mathbb N$ to denote the positive integers $\{1, 2, 3, \dots \}$. A \emph{local ring} for us is a ring with a unique maximal ideal, which is \emph{not necessarily} Noetherian. A ring $S$ is \emph{essentially of finite type} over a ring $R$ if $S$ is the localization of a finitely generated $R$-algebra at some multiplicative set. Given a domain $R$, a \emph{fractional ideal} of $R$ is an $R$-submodule of the fraction field $\text{Frac}(R)$.

Given a field $k$, a \emph{variety} $X$ over $k$ will be an integral, separated scheme of finite type over $k$. We will sometimes write "$X$ is a variety of $K/k$" to mean $X$ is a variety over $k$ with function field $K$. In this paper, the field $k$ is usually perfect of characteristic $p > 0$, and $X$ is usually \emph{regular}, that is, all the local rings $\mathcal{O}_{X,x}$ are regular. Since regular local rings are unique factorization domains, Weil and Cartier divisors coincide on any regular variety. 

Regular varieties $X$ of dimension $n$ over a perfect field are smooth of relative dimension $n$ \cite[$\mathsection$2.2, Proposition 15]{BLR90}. Then $\omega_X \coloneqq \bigwedge^n\Omega_{X/k}$ is a line bundle on $X$, known as the \emph{canonical bundle}, and any Cartier divisor $D$ on $X$ such that 
$\mathcal{O}_X(D) \cong \bigwedge^n\Omega_{X/k}$
is called a \emph{canonical divisor}. The linear equivalence class of $D$ is called the \emph{canonical class}, and denoted $K_X$. By abuse of notation, we often denote a choice of a canonical divisor by $K_X$. Note $\omega_X$ is a dualizing sheaf for $X$. More precisely, if $f: X \rightarrow \Spec(k)$ is the smooth structure map,  $f^!(k[0]) = \omega_X[\text{dim} X]$ is the normalized dualizing complex of $X$ \cite[Chap. V, Theorem 8.3]{RD66}.

\section{Abhyankar valuations}
\label{section-Abhyankar-val}

In this section, we fix a ground field $k$ of arbitrary characteristic, and a finitely generated field extension $K$ of $k$. Additional restrictions on $k$ will be imposed as needed.

\subsection{Background on valuations}
A  \emph{valuation $\nu$ of $K/k$} with \emph{value group} $\Gamma_\nu$ (where $\Gamma_\nu$ is a totally ordered abelian group written additively) is  a surjective group homomorphism
$$\nu\colon K^{\times} \twoheadrightarrow \Gamma_\nu,$$
such that $\nu(k^{\times}) = \{0\}$, and for all $x, y \in K^{\times}$, if $x + y \neq 0$, then $\nu(x+y) \geq \text{min}\{\nu(x), \nu(y)\}$. The \emph{valuation ring} $R_\nu$ of $\nu$ is the ring $\{x \in K^{\times}: \nu(x) \geq 0\} \cup \{0\},$
which is local with maximal ideal
$\mathfrak{m}_\nu  \coloneqq \{x \in K^{\times}: \nu(x) > 0\} \cup \{0\}$
and residue field $\kappa_\nu \coloneqq R_\nu/\mathfrak{m}_\nu$. Both $R_\nu$ and $\kappa_\nu$ are $k$-algebras.

If $(R, \mathfrak m)$ is a local integral domain that is a $k$-algebra with fraction field $K$, we say $\nu$ is \emph{centered} on $R$ if $(R_\nu, \mathfrak{m}_\nu)$ dominates $(R, \mathfrak{m})$, that is, $R \subseteq R_\nu$ and $\mathfrak{m} = \mathfrak{m}_\nu \cap R$. Globally, given a variety $X$ over $k$ with function field $K$, we say  $\nu$ is \emph{centered} on $X$ if the canonical map $\text{Spec}(K) \rightarrow X$ extends to a morphism $\text{Spec}(R_\nu) \rightarrow X$. The image in $X$ of the closed point of $\text{Spec}(R_\nu)$ is called the \emph{center} of $\nu$ on $X$. Note when $X = \operatorname{Spec}(A)$, $\nu$ has a center on $X$ if and only if $A \subseteq R_\nu$. 

We are primarily interested in valuations whose value groups are ordered subgroups of $\mathbb R$, a condition that is equivalent to the valuation rings having Krull dimension $1$ \cite[Theorem 10.7]{Mat89}. For any such real-valued valuation $\nu$ with center $x$ on $X$ and any $m \in \mathbb{R}$, one has the \emph{m-th valuation ideal} $\mathfrak{a}_m(X) \subseteq \mathcal{O}_X$, where locally
\[   
\Gamma(U, \mathfrak{a}_m(X)) = 
     \begin{cases}
       \text{\{$f \in \mathcal{O}_X(U): \nu(f) \geq m$\},} &\quad\text{if $x \in U$,}\\
       \text{$\mathcal{O}_X(U)$,} &\quad\text{if $x \notin U$.} \\ 
     \end{cases}
\]
Note $\mathfrak{a}_m(X) = \mathcal{O}_X$ when $m \leq 0$. If $X = \text{Spec}(A)$, we use $\mathfrak{a}_m(A)$ to denote the ideal $\{a \in A : \nu(a) \geq m\}$ of $A$, and when $X$ or $A$ is clear from context, we just write $\mathfrak{a}_m$. 

An important feature of valuation ideals implicitly used throughout the paper is the following:

\begin{lemma}
\label{valuation-ideals-primary}
Given an affine variety $\Spec(A)$, if $\mathfrak{p}$ is the prime ideal of $A$ corresponding to the center of a real-valued valuation $\nu$ on $\Spec(A)$, then for all real numbers $m > 0$, the ideal $\mathfrak{a}_m(A)$ is $\mathfrak{p}$-primary. Moreover, $\mathfrak{a}_m(A_{\mathfrak p}) = \mathfrak{a}_m(A)A_{\mathfrak p}$.
\end{lemma}

\begin{proof}
For $b \in A$, if $\nu(b) > 0$, then by the Archimedean property, $n\nu(b) = \nu(b^n) \geq m$, for some $n \in \mathbb{N}$. This shows that $\mathfrak{p}$ is the radical of $\mathfrak{a}_m(A)$. If $ab \in \mathfrak{a}_m(A)$ and $a \notin \mathfrak{a}_m(A)$, then $\nu(b) > 0$, so that for some $n$, $b^n \in \mathfrak{a}_m(A)$, as we just showed. Hence $\mathfrak{a}_m(A)$ is $\mathfrak{p}$-primary.

 Note if $s \notin A - \mathfrak{p}$, $\nu(s) = 0$. Thus, the inclusion $\mathfrak{a}_m(A)A_{\mathfrak{p}} \subseteq \mathfrak{a}_m(A_{\mathfrak p})$ is clear. Conversely, if $a/s \in \mathfrak{a}_m(A_{\mathfrak p})$, since $\nu(a/s) = \nu(a) - \nu(s) = \nu(a)$, we get $a \in \mathfrak{a}_m(A)$, proving $\mathfrak{a}_m(A_{\mathfrak p}) \subseteq \mathfrak{a}_m(A)A_{\mathfrak{p}}$.
\end{proof}

\subsection{Abhyankar valuations} 
\label{Abhyankar valuations}
Associated with a valuation $\nu$ of $K/k$ with value group $\Gamma_\nu$ and residue field $\kappa_\nu$ are two important invariants.

\begin{definition}
\label{rational-rank-residue-degree}
The \textbf{rational rank} of $\nu$, abbreviated as $\rr{\nu}$, is by definition the dimension of the $\mathbb{Q}$-vector space $\mathbb{Q} \otimes_{\mathbb Z} \Gamma_\nu$, and the \textbf{transcendence degree of $\nu$}, abbreviated as $\td \nu$, is the transcendence degree of $\kappa_\nu$ over $k$.
\end{definition}

The following fundamental inequality relates the rational rank and transcendence degree of a valuation (see \cite{Abh56} for a generalization):

\begin{theorem}\textup{\cite[VI, $\mathsection$10.3, Corollary 1]{Bou89}}
\label{Abhyankar-inequality}
Let $\nu$ be a valuation of $K/k$, with value group $\Gamma_\nu$ and residue field $\kappa_\nu$. Then
$$\rr{\nu} + \td{\nu} \leq \td{K/k}.$$
If equality holds in the above inequality, then $\Gamma_\nu$ is a finitely generated abelian group (hence $\Gamma_\nu \cong \mathbb{Z}^{\oplus r}$, for $r = \rr{\nu}$), and $\kappa_\nu$ is a finitely generated field extension of $k$.
\end{theorem}

An \textbf{Abhyankar valuation} $\nu$ of $K/k$ is a valuation for which $\rr{\nu} + \td{\nu} = \td{K/k}$.

\begin{examples}
\label{examples-Abhyankar-valuations}
{\*}
\begin{enumerate}
\item (\emph{Prototype}) Let $X$ be a normal variety over $k$ of dimension $n$, and $D$ a prime divisor on X. Then we get a valuation of $K(X)/k$, denoted $\text{ord}_D$, which is called the \emph{order of vanishing along $D$}. The value group of $\text{ord}_D$ is $\mathbb{Z}$, and the valuation ring $\mathcal{O}_{X,D}$ equals the stalk at the generic point of $D$. Then $\text{ord}_D$ is Abhyankar since it is has rational rank $1$ and transcendence degree $n-1$.

\noindent In fact, Zariski showed that if $\nu$ is an Abhyankar valuation of $K/k$ of rational rank $1$, then there exists some normal model $X$ of $K/k$ and a prime divisor $D$ on $X$ such that $\nu$ is given by order of vanishing along $D$ \cite[VI, $\mathsection$14, Theorem 31]{ZS60}.

\item There are real-valued Abhyankar valuations on $\mathbb{F}_p(X, Y)/\mathbb{F}_p$ which are not discrete. For example, let $\alpha$ be any irrational number, and $\Gamma$ the subgroup $\mathbb{Z} \oplus \mathbb{Z}\alpha$ of $\mathbb{R}$ with the induced order. There exists a unique valuation $\nu$ of $\mathbb{F}_p(X, Y)/\mathbb{F}_p$ with value group $\Gamma$ such that
$$\nu_\alpha(X) = 1 \hspace{2mm} \text{and} \hspace{2mm} \nu_\alpha(Y) = \alpha.$$
Then $\nu_\alpha$ is Abhyankar with rational rank $2$ and transcendence degree $0$. Note $\nu_\alpha$ is not a discrete valuation because $\Gamma$ is not isomorphic to $\mathbb Z$.

\item (\emph{Non real-valued Abhyankar valuation}) Consider again the extension $\mathbb{F}_p(X,Y)/\mathbb{F}_p$. Giving $\mathbb{Z} \oplus \mathbb{Z}$ the lexicographical order, there exists a unique valuation $\nu_{lex}$ of $\mathbb{F}_p(X, Y)/\mathbb{F}_p$ with value group $\mathbb Z \oplus \mathbb Z$ such that
$$\nu_{lex}(X) = (1,0) \hspace{2mm} \text{and} \hspace{2mm} \nu_{lex}(Y) = (0,1).$$
Since the rational rank of $\nu_{lex}$ is $2$, it is also Abhyankar. However, there is no order preserving embedding of $\mathbb{Z} \oplus \mathbb{Z}$ with lex order into $\mathbb R$. 

\item (\emph{Non-example}) Take the formal Laurent series field $\mathbb{F}_p((t))$. This has the $t$-adic valuation $\nu_t$ over $\mathbb{F}_p$, whose corresponding valuation ring is the formal power series ring $\mathbb{F}_p[[t]]$. Choose an embedding of fields 
$$\mathbb{F}_p(X, Y) \hookrightarrow \mathbb{F}_p((t))$$
that maps $X \mapsto t$, and $Y$ to $q(t) \in \mathbb{F}_p[[t]]$ such that $\{t, q(t)\}$ are algebraically independent over $\mathbb{F}_p$. Restricting $\nu_t$ to $\mathbb{F}_p(X,Y)$ gives a discrete valuation $\nu_{q(t)}$ of $\mathbb{F}_p(X, Y)/\mathbb{F}_p$, such that 
$$\kappa_{\nu_{q(t)}} = \mathbb{F}_p.$$
This is because $\kappa_{\nu_{q(t)}}$ contains $\mathbb{F}_p$, and is also contained in the residue field of $\mathbb{F}_p[[t]]$, which equals $\mathbb F_p$. Then $\rr{\nu_{q(t)}} = 1$, \text{tr.deg $\nu_{q(t)} = 0$}, and so $\nu_{q(t)}$ is not an Abhyankar valuation.
\end{enumerate}
\end{examples}

\subsection{Local uniformization of Abhyankar valuations}
Local uniformization of a valuation $\nu$ of $K/k$ is a local analogue of resolution of singularities. In its simplest form, it asks if there exists a regular variety of $K/k$ on which $\nu$ admits a center. Zariski solved local uniformization when $k$ has characteristic $0$ \cite{Zar40}, long before Hironaka's seminal work on resolution of singularities, and used it to simplify the proof of resolution of singularities of surfaces \cite{Zar42}. More recently, de Jong's work on alterations \cite{deJ96} showed that local uniformization can always be achieved up to a finite extension of the function field, and that in characteristic $p > 0$, this extension can even be taken to be purely inseparable \cite{Tem13}. However, local uniformization remains elusive in positive characteristic, although for Abhyankar valuations Knaf and Kuhlmann establish the following result.

\begin{theorem}\textup{\cite[Theorem 1.1]{Knaf05}}
\label{local-uniformization-Abhyankar-valuations}
Let $K$ be a finitely generated field extension of any field $k$, and $\nu$ an Abhyankar valuation of $K/k$ with valuation ring $(R_\nu, \mathfrak{m}_\nu, \kappa_\nu)$. Suppose that $\kappa_\nu$ is separable over $k$. Then, given any finite subset $Z \subset R_\nu$, there exists a variety $X$ of $K/k$, and a center $x$ of $\nu$ on $X$ satisfying the following properties:
\begin{enumerate}
\item $\mathcal{O}_{X,x}$ is a regular local ring of dimension equal to the rational rank of $\nu$.
\item $Z \subseteq \mathcal{O}_{X,x}$, and there exists a regular system of parameters $x_1, \dots, x_d$ of $\mathcal{O}_{X,x}$ such that every $z \in Z$ admits a factorization
$$z = ux^{a_1}_1\dots x^{a_d}_d,$$
for some $u \in \mathcal{O}_{X,x}^{\times}$, and $a_i \geq 0$.
\end{enumerate}
\end{theorem}

\begin{remark}
\label{local-uniformization-perfect-ground-field}
Any Abhyankar valuation $\nu$ over a perfect field $k$ satisfies Theorem \ref{local-uniformization-Abhyankar-valuations} since $\kappa_\nu/k$ is then automatically separable. There are other related approaches to local uniformization of Abhyankar valuations. For instance, see \cite[Section 5]{Tem13} and \cite[Corollary 7.25]{Tei14}.                                                                                                                                                                                                                                                                                                                                                                                                                                                                                                                                                                                                                                                                                                                                                                                                                                                                                                                                                                                                                                      \end{remark}

The presence of the set $Z$ in Theorem \ref{local-uniformization-Abhyankar-valuations} allows us to deduce that real-valued Abhyankar valuations over perfect fields of arbitrary characteristic can be locally monomialized. This will be important in the proof of Theorem \ref{Theorem-B}.

\begin{proposition}[\textbf{Local monomialization}]
\label{corollary-theorem-on-local-uniformization}
Assume $k$ is perfect, and $\nu$ is a non-trivial Abhyankar valuation of $K/k$ of rational rank $d$, centered on an affine variety $\Spec(R)$ of $K/k$. Then there exists a variety $\Spec(S)$ of $K/k$, along with an inclusion of rings $R \hookrightarrow S$ satisfying:
\begin{enumerate}
\item[(a)] $\Spec(S)$ is regular, and $\nu$ is centered at $x \in \Spec(S)$ such that $\mathcal{O}_{\Spec(S),x}$ is a regular local ring of dimension $d$, and the induced map of residue fields $\kappa(x) \hookrightarrow \kappa_\nu$ is an isomorphism.
\vspace{0.5mm}

\item[(b)] There exists a regular system of parameters $\{x_1, \dots, x_d\}$ of $\mathcal{O}_{\Spec(S),x}$ such that $\nu(x_1), \dots,$ $\nu(x_d)$ freely generate the value group $\Gamma_\nu$.
\end{enumerate}
If in addition $\nu$ is real-valued, then the valuation ideals of $\mathcal{O}_{\Spec(S),x}$ are generated by monomials in $x_1,\dots,x_d$.
\end{proposition}

\begin{proof}
Since the value group $\Gamma_\nu$ is free of rank $d$ (Theorem \ref{Abhyankar-inequality}), one can choose $r_1, \dots, r_d \in R_\nu$ such that $\nu(r_1), \dots, \nu(r_d)$ freely generate $\Gamma_\nu$. Also, because $\kappa_\nu$ is a finitely generated field extension of $k$, there exist $y_1, \dots, y_j \in R_\nu$ whose images in $\kappa_\nu$ generate $\kappa_\nu$ over $k$. Let $t_1, \dots, t_n \in K$ be generators for $R$ over $k$. Then $t_1, \dots, t_n \in R_\nu$ because $\nu$ is centered on $\text{Spec}(R)$. Defining 
$$Z := \{t_1, \dots, t_n, y_1, \dots, y_j, r_1, \dots, r_d\},$$
by Theorem \ref{local-uniformization-Abhyankar-valuations} there exists a variety $X$ over $k$ with function field $K$ such that $\nu$ is centered at a regular point $x \in X$ of codimension $d$,  $Z \subseteq \mathcal{O}_{X,x}$, and there exists a regular system of parameters $\{x_1, \dots, x_d\}$ of $\mathcal{O}_{X,x}$ with respect to which every $z \in Z$ can be factorized as
$$z = ux^{a_1}_1 \dots x^{a_d}_d,$$
for some $u \in \mathcal{O}_{Y,y}^{\times}$, and integers $a_i \geq 0$. In particular, each $\nu(r_i)$ is a $\mathbb{Z}$-linear combination of $\nu(x_1), \dots, \nu(x_d)$, which shows that $\{\nu(x_1), \dots, \nu(x_d)\}$ also freely generates $\Gamma_\nu$. Moreover, by our choice of $Z$, $\kappa(x) \hookrightarrow \kappa_\nu$ is an isomorphism.

Since $t_1, \dots, t_n \in \mathcal{O}_{X,x}$, we have an inclusion $R \subseteq \mathcal{O}_{X,x}$. Now restricting to an affine neighborhood of $x$, we may assume $X = \text{Spec}(S)$, where $t_1, \dots, t_n \in S$ and $S$ is regular. Then by construction, $R \subseteq S$, and parts (a) and (b) of the Proposition are satisfied.

Suppose $\nu$ is also real-valued, and $\mathfrak{p}$ is the maximal ideal of $\mathcal{O}_{\text{Spec}({S}), x}$. We want to show that the valuation ideals $\mathfrak{a}_m$ of $\mathcal{O}_{\text{Spec}({S}), x}$ are monomial in $\{x_1,\dots,x_d\}$. For $m > 0$, since $\mathfrak{a}_m$ is $\mathfrak{p}$-primary, we know that $\mathfrak{p}^n \subseteq \mathfrak{a}_m$ for some $n \in \mathbb{N}$. Note $\mathfrak{p}^n$ has a monomial generating set $\{x^{\alpha_1}_1\dots x^{\alpha_d}_d: \alpha_1 + \dots + \alpha_d = n\}$. Modulo $\mathfrak{p}^n$, any nonzero element $t \in \mathfrak{a}_m$ can be expressed as a finite sum $s$ of monomials of the form $x^{\beta_1}_1\dots x^{\beta_d}_d$, with $0 < \beta_1 + \dots +\beta_d \leq n-1$, and where the coefficients of the monomials are units in $\mathcal{O}_{\text{Spec}({S}), x}$. Then expressing $t = s + u$, for $u \in \mathfrak{p}^n$, we see that $\nu(s) \geq m$ because both $\nu(t), \nu(u) \geq m$. However, $\nu(s)$ equals the smallest valuation of the monomials $x^{\beta_1}_1\dots x^{\beta_d}_d$ appearing in the sum since monomials have distinct valuations. Thus, each such $x^{\beta_1}_1\dots x^{\beta_d}_d \in \mathfrak{a}_m$, completing the proof.
\end{proof}

\begin{example}
Let $\nu_{\pi}$ be the valuation of $\mathbb{F}_p(X,Y,Z)/\mathbb{F}_p$ with value group $\mathbb{Z} \oplus \mathbb{Z}\pi \subset \mathbb{R}$ such that $\nu_{\pi}(X) = 1 = \nu_{\pi}(Y)$, $\nu_{\pi}(Z) = \pi$, and for any polynomial $\sum b_{\alpha\beta\gamma}X^{\alpha}Y^{\beta}Z^{\gamma} \in \mathbb{F}_p[X,Y,Z]$,
$$\nu_{\pi}(\sum b_{\alpha\beta\gamma}X^{\alpha}Y^{\beta}Z^{\gamma}) = \inf\{\alpha + \beta + \pi\gamma: b_{\alpha\beta\gamma} \neq 0\}.$$
One can verify that $\nu_{\pi}$ is Abhyankar with $\rr{\nu_{\pi}} = 2$ and $\td{\nu_{\pi}} = 1$. For example $Y/X$ is a unit in the valuation ring $R_{\nu_{\pi}}$ whose image in the residue field is transcendental over $\mathbb F_p$. Note $\nu_{\pi}$ is centered on $\mathbb{A}^3_{\mathbb{F}_p}= \Spec(\mathbb{F}_p(X,Y,Z))$ at the origin. However, the system of parameters $X, Y, Z$ of the local ring at the origin do not freely the generate the value group.  On the other hand, blowing up the origin and considering the affine chart $\Spec(\mathbb{F}_p[X, \frac{Y}{X}, \frac{Z}{X}])$, we see that $\nu_{\pi}$ is centered on $\mathbb{F}_p[X, \frac{Y}{X}, \frac{Z}{X}]$ with center $(X, {Z}/{X})$, and now the regular system of parameters $X, {Z}/{X}$ of the local ring $\mathbb{F}_p[X, \frac{Y}{X}, \frac{Z}{X}]_{(X, \frac{Z}{X})}$ do indeed freely generate the value group.  
\end{example}

\section{Characteristic p preliminaries}
\label{Characteristic p preliminaries}

\subsection{The Frobenius Endomorphism}
We fix a prime number $p > 0$. For any ring $R$ of characteristic $p$, we have the \emph{Frobenius map}
$$F: R \rightarrow R$$
that sends $r \mapsto r^p$. We denote the target copy of $R$ with $R$-module structure induced by restriction of scalars via $F$ as $F_*R$, and for $x \in R$, we also sometimes denote the corresponding element of $F_*R$ as $F_*(x)$. Thus, $R$ acts on $F_*R$ by 
$$r\cdot F_*(x) = F_*(r^px).$$
In what follows, both $F$ and its iterates $F^e: R \rightarrow F^e_*R$ will play an important role. Note that the operation $F^e_*$ gives an exact functor from $\text{Mod}_R \rightarrow \text{Mod}_R$ since it is just restruction of scalars.

Globally, if $X$ is a scheme over $\mathbb{F}_p$, then we have the \emph{(absolute) Frobenius endomorphism}
$$F: X \rightarrow X,$$
which on the underlying topological spaces is the identity map, while inducing a map on structure sheaves
$$\mathcal{O}_X \rightarrow F_*\mathcal{O}_X$$
by raising local sections to their $p$-th powers. By a common abuse of notation, we also denote the map on structure sheaves induced by Frobenius by $F: \mathcal{O}_X \rightarrow F_*\mathcal{O}_X$.

We will primarily be interested in the class of schemes for which the Frobenius morphism is finite. These have a special name.

\begin{definition}
A scheme $X$ over $\mathbb F_p$ is \textbf{F-finite} if Frobenius on $X$ is a finite morphism. A ring $R$ is \textbf{F-finite} if $\text{Spec}(R)$ is F-finite, or equivalently, if $F_*R$ is a finitely generated $R$-module.
\end{definition}

F-finite rings and schemes are ubiquitous. Indeed, it is easy to show that any ring essentially of finite type over a perfect (or even F-finite) field is F-finite. This means that all varieties in this paper are F-finite. More generally, finiteness of Frobenius is preserved under finite type ring maps, localization, and completion of Noetherian local rings. F-finite Noetherian rings also satisfy good geometric properties. For example, Kunz showed that F-finite Noetherian rings are excellent \cite{Kun76}, the converse being true under mild restrictions \cite{DS18}.

Quite remarkably, Frobenius is able to detect regularity of a local ring. 

\begin{theorem}\textup{\cite[Theorem 2.1]{Kun69}}
\label{Kunz-regular}
Let $(R, \mathfrak m)$ be a Noetherian local ring of prime characteristic. Then $R$ is regular if and only if the Frobenius map $F: R \rightarrow F_*R$ is flat. 
\end{theorem}

Globally, Kunz's theorem implies that if $X$ is an F-finite, locally Noetherian scheme, then $X$ is regular if and only if $F_*\mathcal{O}_X$ is a locally free sheaf on $X$.
Theorem \ref{Kunz-regular} was the starting point of systematically using the Frobenius map to study singularities in prime characteristic, and the various notions of singularities proposed and studied since Kunz's result (such as F-purity, Frobenius splitting, F-regularity, etc.) try to measure how far $F_*R$ is from being a flat $R$-module.

\subsection{$p^{-e}$-linear maps}
\label{$p^{-e}$-linear maps}
Let $X$ be a scheme over $\mathbb F_p$. For sheaves of $\mathcal{O}_X$-modules $\mathcal F, \mathcal G$, $\mathcal O_X$-linear maps of the form
$$\eta: F^e_*\mathcal{G} \rightarrow \mathcal{F}$$
will play an important role in this paper, especially when we introduce test ideals. Following \cite{BB11}, we call such maps \textbf{$p^{-e}$-linear maps}. The name comes from the fact that if $r \in \mathcal{O}_X(U), g \in F^e_*\mathcal{G}(U) = \mathcal{G}(U)$ are local sections, then
$$\eta(r^{p^e}g) = r\eta(g).$$
For $R$-modules $M, N$ we similarly have $p^{-e}$-linear maps of $R$-modules which are $R$-linear maps
$$F^e_*N \rightarrow M.$$

Perhaps the simplest example of such a map is a Frobenius splitting.

\begin{definition}
\label{F-splitting-defn}
A scheme $X$ over $\mathbb F_p$ is \textbf{Frobenius split} if there exists a $p^{-1}$-linear map $\eta: F_*\mathcal{O}_X \rightarrow \mathcal{O}_X$ such that the composition $\mathcal{O}_X \xrightarrow{F} F_*\mathcal{O}_X \xrightarrow{\eta} \mathcal{O}_X$
is the identity. The map $\eta$ is then called a \textbf{Frobenius splitting of $X$}. A ring $R$ is \textbf{Frobenius split} if $\text{Spec}(R)$ is Frobenius split, that is, the Frobenius map $F: R \rightarrow F_*R$ admits a left inverse in the category of $R$-modules.
\end{definition}

\begin{example}
Let $R = \mathbb{F}_p[x, y]$, where $x, y$ are indeterminates. Then $F_*R$ is a free $R$-module with basis
$$\{x^iy^j: 0 \leq i, j \leq p-1\}.$$
One obtains a Frobenius splitting $F_*R \rightarrow R$ by sending $1 \mapsto 1$ and the other basis elements to $0$.
\end{example}

Schemes with non-trivial $p^{-e}$-linear maps often satisfy surprising cohomological properties. For example, Mehta and Ramanathan, who coined the term \emph{Frobenius splitting}, showed that Frobenius split projective varieties satisfy Kodaira vanishing \cite{MR85}, even though Kodaira vanishing is known to fail in general in characteristic $p$ \cite{Ray78}.

Along similar lines, the following observation of Mehta and Ramanathan, that associates geometric data to $p^{-e}$-linear maps, is at the basis of many local and global results in positive characteristic. We will later use it to examine the behavior of test ideals under birational maps (Proposition \ref{test-ideals-birational-maps}).

\begin{proposition}\textup{\cite{MR85, BS13}}
\label{p-maps-divisors}
Suppose $X$ is a regular variety over a perfect field $k$ of prime characteristic, and $K_X$ is a canonical divisor on $X$. Then for any divisor $D$,
$$\mathcal{H}{om}_{\mathcal O_X}(F^e_*\mathcal{O}_X(D), \mathcal{O}_{X}) \cong F^e_*\mathcal{O}_X((1-p^e)K_X - D).$$
\end{proposition}

\begin{proof}
This result is well-known, and even works when $X$ is normal. A key ingredient is Grothendieck duality for finite maps. For a proof of the isomorphism when $D = 0$, see \cite[Lemma 2.17]{SZ15}. The proof for an arbitrary divisor D follows readily by adapting the argument in \cite{SZ15} (see also \cite[Theorem 4.1.3]{BS13}). 
\end{proof}

\begin{remark}
\label{important remark}
Taking global sections under the isomorphism of Proposition \ref{p-maps-divisors}, we see that nonzero $p^{-e}$-linear maps $F^e_*\mathcal{O}_X(D) \rightarrow \mathcal{O}_X$, up to pre-multiplication by a unit of $\Gamma(X, \mathcal{O}_X)$, are in bijection with effective divisors on $X$ linearly equivalent to $(1-p^e)K_X - D$. More generally, if $K$ is the function field of $X$ regarded as a constant sheaf, then $p^{-e}$-linear maps $F^e_*\mathcal{O}_X(D) \rightarrow K$ correspond to \emph{not necessarily} effective divisors linearly equivalent to $(1-p^e)K_X - D$. The image of $F^e_*\mathcal{O}_X(D) \rightarrow K$ lies in $\mathcal{O}_X$ precisely when the associated divisor is effective \cite[Exercise 4.13]{BS13}.
\end{remark}

The existence of non-trivial $p^{-e}$-linear maps puts restrictions on the geometry of a variety. For example, since a Frobenius splitting is a global section of $\mathcal{H}{om}_{\mathcal{O}_X}(F_*\mathcal{O}_X, \mathcal{O}_X)$, taking $D = 0$ in Proposition \ref{p-maps-divisors} we see that any Frobenius split, regular projective variety over a perfect field cannot have ample $K_X$. For other similar consequences, we recommend the survey \cite{BS13}.

\section{Test Ideals}
\label{Test Ideals}
Beginning with a review of test ideals for pairs, the goal is to construct an asymptotic version that plays a role similar to asymptotic multiplier ideals in characteristic $0$. We also examine how asymptotic test ideals transform under \'etale and birational ring maps. 

\subsection{Summary of test ideal for pairs}
First defined in \cite{HY03}, test ideals of pairs are intimately related to multiplier ideals of pairs. For example, it follows from results in \cite{Smi00, Har01, HY03} that, on reducing multiplier ideals of pairs modulo primes, one gets test ideals of pairs.  However, instead of the tight closure approach of \cite{HY03}, in this paper we adopt a dual point of view due to Schwede \cite{Sch10}, and define test ideals using $p^{-e}$-linear maps. For an excellent synthesis of the various viewpoints on test ideals we suggest the survey \cite{ST12} to the interested reader.

\begin{definition}
\label{test-ideal-pair-definition}
Let $R$ be an F-finite Noetherian domain, $\mathfrak{a} \subseteq R$ a nonzero ideal, and $t > 0$ a real number. The test ideal of the \textbf{pair $(R, \mathfrak{a}^t)$} is defined to be the smallest \emph{nonzero} ideal $J$ of $R$ such that for all $e \in \mathbb{N}$, and $\phi \in \text{Hom}_R(F^e_*R, R)$,
$$\phi(F^e_*(J\mathfrak{a}^{\ceil{t(p^e-1)}})) \subseteq J. \footnote{In tight closure literature, this is usually called the \emph{big} or \emph{non-finitistic} test ideal of the pair $(R, \mathfrak a^t)$.}$$
It is denoted $\tau(R, \mathfrak{a}^t)$, or $\tau(\mathfrak{a}^t)$ when $R$ is clear from context. If $\mathfrak{a} = R$, we usually just write $\tau(R)$, and call it \textbf{the test ideal of $R$}.
\end{definition}

\begin{remark}
\label{test-elements}
The existence of $\tau(R, \mathfrak{a}^t)$, which is not at all obvious from its definition, is a consequence of a deep result of Hochster and Huneke on the existence of \emph{(completely stable) test elements} for Noetherian, F-finite domains. Test elements are abundant in this setting -- any $c \in R$ such that $R_c$ is regular (such c always exist since $R$ is generically regular and excellent) has a power that is a test element \cite[Theorem 5.10]{HH94}. Given a test element $c \in R$, Hara and Takagi show that
$$\tau(\mathfrak{a}^t) = \sum_{e \in \mathbb{N}} \sum_{\phi} \phi(F^e_*(c\mathfrak{a}^{\ceil{t(p^e-1)}})),$$
where $\phi$ ranges over all elements of $\text{Hom}_R(F^e_*R, R)$ \cite[Lemma 2.1]{HT04}.
\end{remark}

Having addressed the issue of the existence of test ideals, we now collect most of their basic properties, in part to highlight their similarity with multiplier ideals.  

\begin{theorem}
\label{basic-properties-test-ideals}
Suppose $R$ is a Noetherian, F-finite domain with nonzero ideals $\mathfrak{a}$, $\mathfrak{b}$. Let $t > 0$ be a real number.
\begin{enumerate}
\item If $\mathfrak{a} \subseteq \mathfrak{b}$, then $\tau(\mathfrak{a}^t) \subseteq \tau(\mathfrak{b}^t)$. 
\item If the integral closures of $\mathfrak{a}$ and $\mathfrak{b}$ coincide, then $\tau(\mathfrak{a}^t) = \tau(\mathfrak{b}^t)$.
\item If $s > t$, then $\tau(\mathfrak{a}^s) \subseteq \tau(\mathfrak{a}^t)$.
\item For any $m \in \mathbb{N}$, $\tau((\mathfrak{a}^m)^t) = \tau(\mathfrak{a}^{mt})$.
\item There exists some $\epsilon > 0$ depending on $t$ such that for all $s \in [t, t + \epsilon]$, $\tau(\mathfrak{a}^s) = \tau(\mathfrak{a}^t)$.
\item $\tau(R)$ defines the closed locus of prime ideals $\mathfrak p$ where $R_{\mathfrak{p}}$ is not strongly F-regular \footnote{We refrain from defining strong F-regularity \cite[pg. 128]{HH89} since we do not need this notion in any essential way. Note that  regular F-finite domains are automatically strongly F-regular.}. Thus, $\tau(R) = R$ if and only if $R$ is strongly F-regular.
\item We have $\tau(R)\mathfrak{a} \subseteq \tau(\mathfrak{a})$. 
Hence, if $R$ is strongly F-regular (in particular regular), $\mathfrak{a} \subseteq \tau(\mathfrak{a})$.
\item If $S \subset R$ is a multiplicative set, then $\tau(S^{-1}R, (\mathfrak{a}S^{-1}R)^t)= \tau(R, \mathfrak{a}^t)S^{-1}R$.
\item If $(R, \mathfrak m)$ is local, and $\widehat{R}$ denotes the $\mathfrak m$-adic completion of $R$, then $\tau(\widehat{R}, (\mathfrak{a}\widehat{R})^t) = \tau(R, \mathfrak{a}^t)\widehat{R}$.
\item If $R$ is regular, $x \in R$ a regular parameter, and $\overline{R}\coloneqq R/xR$, then $\tau(\overline{R}, (\mathfrak{a}\overline{R})^t) \subseteq \tau(R, \mathfrak{a}^t)\overline{R}.$
\item \textbf{(Subadditivity)} If $R$ is regular, then for all $n \in \mathbb N$, $\tau(\mathfrak{a}^{nt}) \subseteq \tau(\mathfrak{a}^t)^n.$
\end{enumerate}
\end{theorem}

\begin{proof}[Indication of proof]
For proofs and precise references for all statements, please consult \cite[Section 6]{ST12}, or \cite[Theorem 4.6]{SZ15} when the ring is regular (the setting of this paper). The most important of the above properties for us is (11), the subadditivity of test ideals. For a slick proof of subadditivity in a slightly more general setting see \cite[Theorem 5.5.8]{Mur19}
\end{proof}

\begin{example}[\textbf{Test ideals of monomial ideals}]
\label{test-ideals-monomials}
Let $\mathfrak{a}$ be a nonzero monomial ideal of the polynomial ring $R = k[x_1, \dots, x_n]$, where $k$ is an $F$-finite field characteristic $p > 0$. For any real number $t > 0$, we let $P(t\mathfrak{a})$ denote the convex hull in $\mathbb{R}^n$ of the set
$$\{(ta_1, \dots, ta_n)\colon x^{a_1}_1\dots x^{a_n}_n \in \mathfrak{a}\},$$
and let $\text{Int}(P(t\mathfrak{a}))$ be the points in the topological interior of this convex hull. Then Hara and Yoshida show \cite[Theorem 4.8]{HY03} using the existence of log resolutions in the toric category that test and multiplier ideals of monomial ideals coincide, and so by \cite{How01} 
$$\tau(\mathfrak{a}^t) = \langle x^{b_1}_1\dots x^{b_n}_n \colon b_i \in \mathbb{N} \cup \{0\}, (b_1 + 1, \dots, b_n + 1) \in \text{Int}(P(t\mathfrak{a})) \rangle.$$
\end{example}

\subsection{Asymptotic test ideals}
\label{subsection-asymptotic-test-ideals}
Asymptotic test ideals are defined for graded families of ideals, which we introduce first.

\begin{definition}
\label{graded-family-ideals}
Let $\Phi$ be an additive sub-semigroup of $\mathbb{R}$, and $R$ be a ring. A \textbf{graded family of ideals of R indexed by $\Phi$} is a family of ideals $\{\mathfrak{a}_s\}_{s \in \Phi}$ such that for all $s, t \in \Phi$,
$$\mathfrak{a}_s\cdot \mathfrak{a}_t \subseteq \mathfrak{a}_{s+t}.$$
We also assume $\mathfrak{a}_s \neq 0$, for all $s$.
\end{definition}

\begin{examples}
{\*}
\label{examples of graded families}
\begin{enumerate}
\item If $\mathfrak{a}$ is a nonzero ideal of a domain $R$, then $\{\mathfrak{a}^n\}_{n \in \mathbb N \cup\{0\}}$ is a graded family of ideals.

\item If $R$ is a Noetherian domain, the symbolic powers $\{\mathfrak{a}^{(n)}\}_{n \in \mathbb N \cup \{0\}}$ of a fixed nonzero ideal $\mathfrak{a}$ is an example of a graded family that was studied extensively in \cite{ELS01, HH02}.

\item Let $\nu$ be a non-trivial real-valued valuation of $K/k$ centered on a domain $R$ over $k$ with fraction field $K$. Then the collection of valuation ideals $\{\mathfrak{a}_m(R)\}_{m \in \mathbb R_{\geq 0}}$ is a graded family of ideals by properties of a valuation (since $\nu$ is non-trivial, the ideals $\mathfrak{a}_m$ are all nonzero).
\end{enumerate}
\end{examples}

Now suppose $R$ is an F-finite, Noetherian domain of characteristic $p > 0$, and that $\{\mathfrak{a}_m\}_{m \in \Phi}$ is  a graded family of ideals of $R$, indexed by some sub-semigroup $\Phi$ of $\mathbb R$. Then for any real number $t > 0$, $m \in \Phi$, and $\ell \in \mathbb N$, we have
$$\tau(\mathfrak{a}^t_m) = \tau((\mathfrak{a}^{\ell}_m)^{t/\ell}) \subseteq \tau(\mathfrak{a}^{t/\ell}_{\ell m}).$$
Here the first equality follows from Theorem \ref{basic-properties-test-ideals}(4), and the inclusion follows from Theorem \ref{basic-properties-test-ideals}(1) using the fact that $\mathfrak{a}^{\ell}_m \subseteq \mathfrak{a}_{\ell m}$. 

Thus, for any $m \in \Phi$, the set $\{\tau(\mathfrak{a}^{1/\ell}_{\ell m})\}_{\ell \in \mathbb N}$ is filtered under inclusion (for instance, $\tau(\mathfrak{a}^{1/\ell_1}_{\ell_1 m})$ and $\tau(\mathfrak{a}^{1/\ell_2}_{\ell_2 m})$ are both contained in $\tau(\mathfrak{a}^{1/\ell_1\ell_2}_{\ell_1\ell_2 m}$)). Since $R$ is a Noetherian ring, this implies that $\{\tau(\mathfrak{a}^{1/\ell}_{\ell m})\}_{\ell \in \mathbb N}$ has a unique maximal element under inclusion, which will be the $m$-th asymptotic test ideal.

\begin{definition}
\label{asymptotic-test-ideal-definition}
For a graded family of ideals $\mathfrak{a}_{\bullet} = \{\mathfrak{a}_{m}\}_{m \in \Phi}$ of a Noetherian, F-finite domain $R$, and for any $m \in \Phi$, we define the \textbf{$m$-th asymptotic test ideal of the graded system}, denoted $\tau_m(R, \mathfrak{a}_{\bullet})$ (or $\tau_m(\mathfrak{a}_{\bullet})$ when $R$ is clear from context), as follows:
$$\tau_m(R, \mathfrak{a}_{\bullet}) \coloneqq \sum_{\ell \in \mathbb N} \tau(\mathfrak{a}^{1/\ell}_{\ell m}).$$
By the above discussion, $\tau_m(R, \mathfrak{a}_{\bullet})$ equals $\tau(\mathfrak{a}^{1/\ell}_{\ell m})$, for some highly divisible $\ell \gg 0$.
\end{definition}

Asymptotic test ideals satisfy appropriate analogues of properties satisfied by test ideals of pairs (Theorem \ref{basic-properties-test-ideals}), since they equal test ideals of suitable pairs. We highlight a few properties that will be important for us in the sequel.

\begin{proposition}\textup{\cite{Har05, SZ15}}
\label{some-properties-asymptotic-test-ideal}
Suppose $R$ is a regular domain, essentially of finite type over a perfect field of positive characteristic, with a graded family of ideals $\mathfrak{a}_{\bullet} = \{\mathfrak{a}_m\}_{m \in \Phi}$. We have the following:
\begin{enumerate}
\item For any $m \in \Phi$, 
$\mathfrak{a}_m \subseteq \tau(\mathfrak{a}_m) \subseteq \tau_m(\mathfrak{a}_{\bullet}).$
\item For any $m \in \Phi$, $\ell \in \mathbb N$,
$\mathfrak{a}_{\ell m} \subseteq \tau_{\ell m}(\mathfrak{a}_{\bullet}) \subseteq \tau_m(\mathfrak{a}_{\bullet})^{\ell}.$
\end{enumerate}
\end{proposition}

\begin{proof}
We get (1) using Theorem \ref{basic-properties-test-ideals}(7), and the definition of asymptotic test ideals.

Property (2) is crucial, and is a consequence of the subadditivity property of test ideals (Theorem \ref{basic-properties-test-ideals}(11)). The first inclusion $\mathfrak{a}_{\ell m} \subseteq \tau_{\ell m}(\mathfrak{a}_{\bullet})$ follows from (1). For the second inclusion, for a highly divisible $n \gg 0$, we have
$$\tau_{\ell m}(\mathfrak{a}_{\bullet}) = \tau(\mathfrak{a}^{1/n}_{n \ell m}) = \tau(\mathfrak{a}_{n\ell m}^{\ell/n\ell}),$$
and by subadditivity, $\tau(\mathfrak{a}^{\ell/n\ell}_{n\ell m}) \subseteq \tau(\mathfrak{a}^{1/n\ell }_{n\ell m})^{\ell} = \tau_m(\mathfrak{a}_{\bullet})^{\ell}$, completing the proof.
\end{proof}

\subsection{(Asymptotic) test ideals and \'etale maps}
\label{test-ideals-etale-maps-subsection}
In this subsection, we study a transformation law for test ideals under essentially \'etale maps. We say that a local homomorphism of Noetherian local rings 
$$\varphi: (A, \mathfrak m) \rightarrow (B, \mathfrak n)$$ 
is an \textbf{\'etale local homomorphism of local rings} in the sense of \cite[\href{https://stacks.math.columbia.edu/tag/0258}{Tag 0258}]{Stacks} if $\varphi$ is flat and unramified (i.e. $\mathfrak{m}_AB = \mathfrak{m}_B$, $\kappa_A \hookrightarrow \kappa_B$ is finite separable) and $\varphi$ is essentially of finite type.

\begin{definition}
We say that a homomorphism $f: R \rightarrow S$ of Noetherian rings, not necessarily local, is \textbf{essentially \'etale} if $f$ is essentially of finite type, and for all primes ideal $\mathfrak{q} \in \Spec(S)$, the induced map of local rings $R_{f^{-1}(\mathfrak q)} \rightarrow S_{\mathfrak q}$ is an \'etale local homomorphism of local rings in the sense defined above.
\end{definition}

The main result of this subsection is the following transformation law of test ideals under essentially \'etale maps, which although well-known to experts, often appears in the literature with additional finiteness hypotheses  (see \cite{BS02}, \cite{HT04}, \cite{ST14}). We provide a detailed proof because as far as we can tell, the published proofs of this result are all written in the language of tight closure, which differs from our setting.

\begin{proposition}
\label{test-ideals-essentially-etale}
Let $R \rightarrow S$ be an essentially \'etale map of $F$-finite, Noetherian domains. Then for any nonzero ideal $\mathfrak a$ of $R$, and real number $t > 0$,
\begin{equation}
\label{etale-law}
\tau(S, (\mathfrak{a}S)^t) = \tau(R,\mathfrak{a}^t)S.
\end{equation}
\end{proposition}

\begin{proof}
Since the formation of test ideals commutes with localization (Theorem \ref{basic-properties-test-ideals}(8)), and since (\ref{etale-law}) can be checked locally at every prime of $S$, we may assume $S$, and consequently, $R$, are local and $(R, \mathfrak m) \rightarrow (S, \mathfrak n)$ is an \'etale homomorphism of local rings. We also know that the formation of test ideals commutes with completion (Theorem \ref{basic-properties-test-ideals}(9)), and the map from a Noetherian local ring to its completion is faithfully flat. Thus, we can further check (\ref{etale-law}) after completing $R$ and $S$ at their respective maximal ideals. However, the completion of an \'etale local homomorphism of Noetherian local rings is an honest \'etale homomorphism \cite[\href{https://stacks.math.columbia.edu/tag/039M}{Tag 039M}]{Stacks} (it is even finite \'etale, but we do not need finiteness). Hence we may assume that $(R,\mathfrak m) \rightarrow (S, \mathfrak n)$ is an \'etale map. 

An observation of \cite[Remark 6.5]{BS02} shows that there exists $c \in R$ such that $c$ is simultaneously a test element of $R$ and $S$. Indeed, choose $g \in R$ such that $R_g$ is regular. Since $R_g \rightarrow S_g$ is \'etale, $S_g$ is also regular, and so, some power of $g$ is a test element for both $R$ and $S$ (see Remark \ref{test-elements}). Choosing a common test element $c$, we then have
\begin{equation}
\label{eq1}
\tau(S, (\mathfrak{a}S)^t) = \sum_e \sum_{\psi} \psi\big{(}F^e_*(c(\mathfrak{a}S)^{\ceil{t(p^e-1)}})\big{)},
\end{equation}
where $\psi$ ranges over elements of $\textrm{Hom}_S(F^e_*S,S)$, and
\begin{equation}
\label{eq2}
\tau(R,\mathfrak{a}^t)S = \sum_e \sum_{\phi} \phi\big{(}F^e_*(c\mathfrak{a}^{\ceil{t(p^e-1)}})\big{)S},
\end{equation}
where $\phi$ ranges over elements of $\textrm{Hom}_R(F^e_*R,R)$. Thus it suffices to show that the right hand sides of the above two equalities coincide.

Since $R \rightarrow S$ is \'etale, the relative Frobenius map
$$F^e_{S/R}: F^e_*R \otimes_R S \rightarrow F^e_*S$$
that sends $r \otimes s \mapsto r s^{p^e}$ is an isomorphism \cite[XV, Proposition 2(c)(2)]{SGA5}. Consequently, because $R \rightarrow S$ is faithfully-flat and $F^e_*R$ is is a finitely presented $R$-module, we have a canonical isomorphism
$$S \otimes_R \textrm{Hom}_R(F^e_*R, R) \xrightarrow{\sim} \textrm{Hom}_S(F^e_*S,S).$$
The upshot of this discussion is that every $R$-linear map $F^e_*R \rightarrow R$ extends to an $S$-linear map $F^e_*S \rightarrow S$, and every $S$-linear map $F^e_*S \rightarrow S$ is obtained from such an extension. This, combined with (\ref{eq1}) and (\ref{eq2}), clearly implies that 
$$\tau(R, \mathfrak{a}^t)S \subseteq \tau(S, \mathfrak{a}S^t).$$
Conversely, given $\psi: F^e_*S \rightarrow S$ for some $e > 0$, it suffices to show that an element $\gamma \in c(\mathfrak{a}S)^{\ceil{t(p^e-1)}} = c\mathfrak{a}^{\ceil{t(p^e-1)}}S$ of the form
$$\gamma = cas,$$
for $a \in \mathfrak{a}^{\ceil{t(p^e-1)}}$ and $s \in S$ satisfies
$$\psi(\gamma) \in \tau(R,\mathfrak{a}^t)S.$$
Since the relative Frobenius $F_{S/R}$ is surjective, $s$ can be written in the form
$$s = r_1s_1^{p^e} + \dots + r_ns_n^{p^e},$$
for $r_i \in R$ and $s_i \in S$. Thus,
$$\psi(\gamma) = \sum_{i=1}^n \psi(car_is_i^{p^e}) = \sum_{i = 1}^n s_i\psi(car_i),$$
where the second equality follows because $\psi$ is a $p^{-e}$-linear map of $S$-modules. Now choose $\phi \in \textrm{Hom}_R(F^e_*R,R)$ such that $\psi$ is obtained as an extension of $\phi$. Since $car_i \in R$, it follows that
$$\psi(\gamma) = \sum_{i = 1}^n s_i\psi(car_i) = \sum_{i=1}^n s_i\phi(car_i) \in \phi\big{(}F^e_*(c\mathfrak{a}^{\ceil{t(p^e-1)}})\big{)}S \subseteq \tau(R,\mathfrak{a}^t)S,$$
as desired.
\end{proof}

\begin{remarks}
{\*}
\begin{enumerate}
\item A different proof of Proposition \ref{test-ideals-essentially-etale} appears in \cite[Corollary 6.19]{AX16} via properties of the six-functor formalism under the more restrictive hypothesis that the base ring $R$ is Gorenstein.

\item Proposition \ref{test-ideals-essentially-etale} is stated for domains because we made the blanket assumption while developing the theory of test ideals that all rings are domains. But the proof of Proposition \ref{test-ideals-essentially-etale} works for reduced $F$-finite rings as well, primarily because completely stable test elements in the sense of Remark \ref{test-elements} exist for reduced $F$-finite rings \cite[Theorem 5.10]{HH94}.
\end{enumerate}
\end{remarks}

Proposition \ref{test-ideals-essentially-etale} has the following consequence for asymptotic test ideals:

\begin{corollary}
\label{asymptotic-test-ideals-etale-maps}
Let $R \xrightarrow{\varphi} S$ be an essentially \'etale map of $F$-finite, Noetherian domains. Suppose $\mathfrak{a}_{\bullet} = \{\mathfrak{a}_m\}_{m \in \Phi}$ is a graded family of nonzero ideals of $R$, and consider the family $\mathfrak{a}_{\bullet}S = \{\mathfrak{a}_mS\}_{m \in \Phi}$. 
\begin{enumerate}
\item For all $m \in \Phi$, $\tau_m(S, \mathfrak{a}_{\bullet}S) = \tau_m(R, \mathfrak{a}_{\bullet})S.$

\vspace{1mm}

\item If $\bigcap_{m \in \Phi}(\mathfrak{a}_m: \tau_m(R, \mathfrak{a}_{\bullet})) \neq (0)$, then $\bigcap_{m \in \Phi}(\mathfrak{a}_mS: \tau_m(S, \mathfrak{a}_{\bullet}S)) \neq (0)$.
\end{enumerate}
\end{corollary}

\begin{proof}
Again, by the injectivity of $\varphi$, $\mathfrak{a}_{\bullet}S$ is a graded family of nonzero ideals of $S$. Then 
\begin{align*}
\tau_m(S, \mathfrak{a}_{\bullet}S) \coloneqq \sum_{\ell \in \mathbb{N}} \tau\big{(}(\mathfrak{a}_{\ell m}S)^{1/\ell}\big{)} = \sum_{\ell \in \mathbb{N}} \tau(\mathfrak{a}^{1/\ell}_{\ell m})S 
= \bigg{(}\sum_{\ell \in \mathbb{N}} \tau(\mathfrak{a}^{1/\ell}_{\ell m})\bigg{)}S = \tau_m(R, \mathfrak{a}_{\bullet})S,
\end{align*}
where the second quality follows from Proposition \ref{test-ideals-essentially-etale}. This proves (1).

For (2), if $r$ is a nonzero element in $\bigcap_{m \in \Phi}(\mathfrak{a}_m: \tau_m(R, \mathfrak{a}_{\bullet}))$, then using (1), $\varphi(r)$ is a nonzero element in $\bigcap_{m \in \Phi}(\mathfrak{a}_mS: \tau_m(S, \mathfrak{a}_{\bullet}S))$. Note $\varphi(r)$ is nonzero by flatness of $\varphi$ and the fact that $R$ is a domain.
\end{proof}

\subsection{(Asymptotic) test ideals and birational maps}
\label{(Asymptotic) test ideals and birational maps}
We now examine the behavior of test ideals under birational ring maps. The main result (Proposition \ref{test-ideals-birational-maps}) is probably known to experts, but we include a proof, drawing inspiration from \cite{HY03, BS13, ST14}.

\begin{setup}
\label{setup}
Let $k$ be a perfect field of characteristic $p > 0$. Fix an extension $R \hookrightarrow S$ of finitely generated $k$-algebra regular domains such that $\text{Frac}(R) = \text{Frac}(S) = K$. Let $Y = \text{Spec}(S)$, $X = \text{Spec}(R)$, and
$$\pi: Y \rightarrow X$$
denote the birational morphism induced by the extension $R \subseteq S$. Choose canonical divisors $K_Y$ and $K_X$ that agree on the locus where $\pi$ is an isomorphism, and let $K_{Y/X} \coloneqq K_Y - \pi^*K_X$.
Define $\omega_{S/R} \coloneqq \Gamma(Y, \mathcal{O}_Y(K_{Y/X}))$. Then $\omega_{S/R}$ is a locally principal invertible fractional ideal of $S$, with inverse $\omega_{S/R}^{-1} = \Gamma(Y, \mathcal{O}_Y(-K_{Y/X}))$. 
\end{setup}

We use the following fact implicitly in the results of this subsection: \emph{In Setup \ref{setup}, if $\mathfrak{I}$ is a nonzero fractional ideal of $S$, then $R \cap \mathfrak{I}$ is a nonzero ideal of $R$.} 

This follows by clearing the denominator of a nonzero element of $\mathfrak{J}$.

\begin{proposition}
\label{test-ideals-birational-maps}
In Setup \ref{setup}, if $\mathfrak{a}$ is a nonzero ideal of $S$, and $\tilde{\mathfrak{a}}$ denotes its contraction $\mathfrak{a} \cap R$, then for any real $t > 0$,
$$\tau(R, \tilde{\mathfrak{a}}^t) \subseteq \big{(}\omega_{S/R}\cdot\tau(S, (\tilde{\mathfrak{a}}S)^t)\big{)} \cap R \subseteq \big{(}\omega_{S/R}\cdot\tau(S, \mathfrak{a}^t)\big{)} \cap R.$$
\end{proposition} 

\begin{proof}
The inclusion $\big{(}\omega_{S/R}\cdot\tau((\tilde{\mathfrak{a}}S)^t)\big{)} \cap R \subseteq \big{(}\omega_{S/R}\cdot\tau(\mathfrak{a}^t)\big{)} \cap R$ is a consequence of $\tau((\tilde{\mathfrak{a}}S)^t) \subseteq \tau(\mathfrak{a}^t)$ (Theorem \ref{basic-properties-test-ideals}(1)). 

By definition, $\tau(R, \tilde{\mathfrak{a}}^t)$ is the smallest nonzero ideal $J$ of $R$ under inclusion such that for all $e \in \mathbb N$, $\phi \in \text{Hom}_R(F^e_*R, R)$,
$$\phi\big{(}F^e_*(J\tilde{\mathfrak{a}}^{\ceil{t(p^e-1)}})\big{)} \subseteq J.$$
Thus to finish the proof, it suffices to show the above containment for $J = (\omega_{S/R}\cdot\tau((\tilde{\mathfrak{a}}S)^t)) \cap R$. In fact, extending $\phi$ to a $K$-linear map 
$$\phi_K: F^e_*K \rightarrow K,$$
it is enough to show that
\begin{equation}
\label{crucial-containment-ii}
\phi_K\big{(}F^e_*{(}\omega_{S/R}\cdot\tau((\tilde{\mathfrak{a}}S)^t)\cdot\tilde{\mathfrak{a}}^{\ceil{t(p^e-1)}}{)}\big{)} \subseteq \omega_{S/R}\cdot\tau((\tilde{\mathfrak{a}}S)^t).
\end{equation}
Our strategy will be to obtain an $S$-linear map $F^e_*S \rightarrow S$ from $\phi_K$, and then use the defining property of $\tau((\tilde{\mathfrak{a}}S)^t)$ to prove (\ref{crucial-containment-ii}). 

By Proposition \ref{p-maps-divisors}, $\phi$ corresponds to a section $g \in \Gamma(X, \mathcal{O}_X((1-p^e)K_X))$, and then the pullback $\pi^*g$ is a global section of 
$\mathcal{O}_Y((1-p^e)\pi^*K_X) = \mathcal{O}_Y((1-p^e)(K_Y -K_{Y/X})).$
Using Proposition \ref{p-maps-divisors} again, $\pi^*g$ corresponds to a $p^{-e}$-linear map of $\mathcal{O}_Y$-modules $F^e_*\mathcal{O}_Y((1-p^e)K_{Y/X}) \rightarrow \mathcal{O}_Y,$
which on taking global sections gives an $S$-linear map
$$\varphi_g: F^e_*(\omega_{S/R}^{1- p^e}) \rightarrow S.$$
The map $\varphi_g$ can be constructed from $\phi$ in a natural way. For ease of notation,  let 
$$M \coloneqq F^e_*(\omega_{S/R}^{1-p^e}).$$
We claim that algebraically, $\varphi_g$ is obtained by restricting $\phi_K$ to the $S$-submodule $M$ of $F^e_*K$, but this needs justification because $\phi_K|_M$ is \emph{a priori} an $S$-linear map from $M \rightarrow K$, while $\varphi_g$ maps into S. However, choosing a nonzero $f \in R$ such that $R_f \hookrightarrow S_f$ is an isomorphism, we see that on localizing at $f$, the extensions $\varphi_g[f^{-1}]$ of $\varphi_g$ and $\phi_K|_M[f^{-1}]$ of $\phi_K|_M$ agree on the $S$-module $M_f = F^e_*(S_f) = F^e_*(R_f)$ with the map $\phi[f^{-1}]$. Thus, $\varphi_g$ and $\phi_K|_M$ coincide on $M = F^e_*(\omega_{S/R}^{1-p^e})$, and so $\phi_K|_M$ maps into $S$ because $\varphi_g$ does.

Since the inclusion $\tau(R, \tilde{\mathfrak{a}}^t) \subseteq \omega_{S/R}\cdot\tau((\tilde{\mathfrak{a}}S)^t)$ can be checked locally on $S$, one may assume that $\omega^{-1}_{S/R}$ is principal, say 
$\omega_{S/R}^{-1} = cS.$
Then left-mutiplication by $F^e_*(c^{p^e-1})$ induces an $S$-linear map $F^e_*S \rightarrow M$, yielding on composition an element
$$\widetilde{\phi} \coloneqq F^e_*S \xrightarrow{F^e_*(c^{p^e-1}) \cdot} M \xrightarrow{\phi_K|_M} S$$
of $\text{Hom}_S(F^e_*S, S)$. Using $\omega_{S/R} = c^{-1}S$, we finally get
\begin{align*}
\phi_K\big{(}F^e_*(\omega_{S/R}\cdot\tau((\tilde{\mathfrak{a}}S)^t)\cdot\tilde{\mathfrak{a}}^{\ceil{t(p^e-1)}})\big{)} &= c^{-1}\cdot \phi_K\big{(}F^e_*(c^{p^e-1}\tau((\tilde{\mathfrak{a}}S)^t)\cdot\tilde{\mathfrak{a}}^{\ceil{t(p^e-1)}})\big{)} =\\
 c^{-1}\cdot \widetilde{\phi}\big{(}F^e_*(\tau((\tilde{\mathfrak{a}}S)^t) \cdot\tilde{\mathfrak{a}}^{\ceil{t(p^e-1)}})\big{)} & \subseteq c^{-1}\tau((\tilde{\mathfrak{a}}S)^t) = \omega_{S/R} \cdot \tau((\tilde{\mathfrak{a}}S)^t),
\end{align*}
where the inclusion follows from the definition of $\tau((\tilde{\mathfrak{a}}S)^t)$, and the fact that $\widetilde{\phi}\in \text{Hom}_S(F^e_*S, S)$.
\end{proof}

\begin{corollary}
\label{asymptotic-test-ideals-birational-maps}
Suppose in Setup \ref{setup}, we are given a graded family $\mathfrak{a}_{\bullet} = \{\mathfrak{a}_m\}_{m \in \Phi}$ of nonzero ideals of $S$. Denote by $\tilde{\mathfrak{a}}_{\bullet}$ the family $\{\mathfrak{a}_m \cap R\}_{m \in \Phi}$. Then
\begin{enumerate}
\item For all $m \in \Phi$, $\tau_m(R, \tilde{\mathfrak{a}}_{\bullet}) \subseteq (\omega_{S/R}\cdot \tau_m(S, \mathfrak{a}_{\bullet})) \cap R$.
\item If $\bigcap_{m \in \Phi} (\mathfrak{a}_m: \tau_m(S, \mathfrak{a}_{\bullet})) \neq (0)$, then $\bigcap_{m \in \Phi} (\mathfrak{a}_m \cap R: \tau_m(R, \tilde{\mathfrak{a}}_{\bullet})) \neq (0)$.
\end{enumerate}
\end{corollary}

\begin{proof}
Clearly $\tilde{\mathfrak{a}}_{\bullet}$ is a graded family of nonzero ideals of $R$. Now (1) follows from Proposition \ref{test-ideals-birational-maps} by choosing $\ell \gg 0$ such that $\tau_m(\tilde{\mathfrak{a}}_{\bullet}) = \tau((\mathfrak{a}_{\ell m} \cap R)^{1/\ell})$, and $\tau_m(\mathfrak{a}_{\bullet}) = \tau(S, \mathfrak{a}_{\ell m}^{1/\ell})$.

For (2), let $\mathfrak{J}$ denote the nonzero ideal $\bigcap_{m \in \Phi} (\mathfrak{a}_m: \tau_m(\mathfrak{a}_{\bullet}))$ of $S$. Note $\mathfrak{J} \cdot \omega_{S/R}^{-1} \cap R$ is a nonzero ideal of $R$ because $\mathfrak{J} \cdot \omega_{S/R}^{-1}$ is a nonzero fractional ideal of $S$, and $R$ and $S$ have the same fraction field. Then for all $m \in \Phi$,
\begin{align*}
(\mathfrak{J} \cdot \omega_{S/R}^{-1} \cap R) \cdot \tau_m(\tilde{\mathfrak{a}}_{\bullet}) \subseteq {(}\mathfrak{J} \cdot \omega_{S/R}^{-1} \cap R)\big{(}(\omega_{S/R}\cdot \tau_m(\mathfrak{a}_{\bullet})) \cap R\big{)}\\
\subseteq \big{(}\mathfrak{J}\cdot \omega_{S/R}^{-1}\cdot \omega_{S/R} \cdot \tau_m(\mathfrak{a}_{\bullet})\big{)} \cap R = (\mathfrak{J}\cdot\tau_m(\mathfrak{a}_{\bullet}))\cap R \subseteq \mathfrak{a}_m \cap R.
\end{align*} 
Thus, $(0) \neq \mathfrak{J} \cdot \omega_{S/R}^{-1} \cap R \subseteq \bigcap_{m \in \Phi} (\mathfrak{a}_m \cap R: \tau_m(\tilde{\mathfrak{a}}_{\bullet}))$.
\end{proof}

\section{Proof of Theorem \ref{Theorem-B}}
\label{section-proof-theorem-B}

For a ring $A$ of $K/k$ admitting a center of $\nu$, we will say $A$ \textbf{satisfies Theorem \ref{Theorem-B} for $\nu$} if $\bigcap_{m \in \mathbb{R}_{\geq 0}} (\mathfrak{a}_m: \tau_m(\mathfrak{a}_{\bullet})) \neq (0)$, where $\mathfrak{a}_m$ are the valuation ideals of $A$ associated to $\nu$.

To prove Theorem \ref{Theorem-B} we need the following general fact about primary ideals in a Noetherian domain, which in particular implies that if Theorem \ref{Theorem-B} holds for the local ring at the center $x$ of a variety $X$ of $K/k$, then it also holds on any affine open neighborhood of $x$.

\begin{lemma}
\label{primary-ideals-Noetherian domain}
Let $A$ be a Noetherian domain, and $\mathfrak p$ a prime ideal of $A$.
\begin{enumerate}
\item For any $\mathfrak p$-primary ideal $\mathfrak a$ of $A$, $\mathfrak{a} A_{\mathfrak p} \cap A = \mathfrak a$.
\item Let $\{\mathfrak{a}_i\}_{i \in I}, \{J_i\}_{i \in I}$ be collections ideals of $A$ such that each $\mathfrak{a}_i$ is $\mathfrak p$-primary. Then 
$$\bigcap_{i \in I} (\mathfrak{a}_iA_{\mathfrak{p}}: J_iA_{\mathfrak{p}}) =\bigg(\bigcap_{i \in I} (\mathfrak{a}_i: J_i)\bigg{)}A_{\mathfrak p}.$$
Thus, $\bigcap_{i \in I} (\mathfrak{a}_iA_{\mathfrak{p}}: J_iA_{\mathfrak{p}}) \neq (0)$ if and only if $\bigcap_{i \in I} (\mathfrak{a}_i: J_i) \neq (0)$.
\end{enumerate}
\end{lemma}

\begin{proof} [\textbf{Proof of Lemma \ref{primary-ideals-Noetherian domain}}]
(1) follows easily from the definition of a primary ideal. For (2), the containment $\big(\bigcap_{i \in I} (\mathfrak{a}_i: J_i)\big{)}A_{\mathfrak p} \subseteq \bigcap_{i \in I} (\mathfrak{a}_iA_{\mathfrak{p}}: J_iA_{\mathfrak{p}})$ is easy to verify. Now let 
$$\tilde{s} \in \bigcap_{i \in I} (\mathfrak{a}_iA_{\mathfrak{p}}: J_iA_{\mathfrak{p}}),$$ 
and choose $t \in A - \mathfrak p$ such that $t\tilde{s} \in A$, noting that $t\tilde{s}$ is also in the ideal $\bigcap_{i \in I} (\mathfrak{a}_iA_{\mathfrak{p}}: J_iA_{\mathfrak{p}})$. Then for all $i \in I$,
$$(t\tilde{s}) \cdot J_i \subseteq (t\tilde{s})\cdot(J_iA_{\mathfrak p} \cap A) \subseteq \mathfrak{a}_iA_{\mathfrak p} \cap A = \mathfrak{a}_i,$$
where the last equality comes from (1). Thus, $t\tilde{s} \in \bigcap_{i \in I} (\mathfrak{a}_i: J_i)$, and so $\tilde{s} \in \big(\bigcap_{i \in I} (\mathfrak{a}_i: J_i)\big{)}A_{\mathfrak p}$, establishing the other inclusion. Since $A \rightarrow A_{\mathfrak p}$ is injective, the final statement is clear.
\end{proof}

Using Lemma \ref{primary-ideals-Noetherian domain}, Theorem \ref{Theorem-B} is proved as follows:

\begin{proof}[\textbf{Proof of Theorem \ref{Theorem-B}}]
Let $(A, \mathfrak m, \kappa_A)$ be the regular local ring $\nu$ is centered on, where $A$ is essentially of finite type over the perfect field $k$ with fraction field $K$. Suppose $\text{dim}_{\mathbb{Q}} (\mathbb{Q} \otimes_{\mathbb Z} \Gamma_\nu) = d$ and $\td {K/k} = n$. Let $R$ be a finitely generated, regular $k$-subalgebra of $K$ with a prime ideal $\mathfrak{p}$ such that $A = R_{\mathfrak{p}}$. Using local monomialization (Proposition \ref{corollary-theorem-on-local-uniformization}), choose a finitely generated, regular $k$-subalgebra $S$ of $K$ along with an inclusion $R \hookrightarrow S$ such that $\nu$ is centered on the prime $\mathfrak{q}$ of $S$, and $S_{\mathfrak{q}}$ has Krull dimension $d$, with a regular system of parameters $\{x_1, \dots, x_d\}$ such that $\nu(x_1), \dots, \nu(x_d)$ freely generate the value group $\Gamma_\nu$. Note that if $\{\mathfrak{b}_m\}_{m \in \mathbb{R}_{\geq 0}}$ is the set of valuation ideals of $S$, then $\{\mathfrak{b}_m \cap R\}_{m \in \mathbb{R}_{\geq 0}}$ is the set of valuation ideals of $R$. Now if $S_{\mathfrak{q}}$ satisfies Theorem \ref{Theorem-B}, then so does $S$ (Lemma \ref{valuation-ideals-primary} and Lemma \ref{primary-ideals-Noetherian domain}), hence $R$ (Corollary \ref{asymptotic-test-ideals-birational-maps}), hence also $R_{\mathfrak p} = A$ because $\mathfrak{p}$ is the center of $\nu$ on $R$ (using Lemma \ref{primary-ideals-Noetherian domain} again). Thus, it suffices to prove Theorem \ref{Theorem-B} for $A = S_{\mathfrak q}$.

The valuation ideals $\mathfrak{a}_{\bullet} = \{\mathfrak{a}_m\}_{m \in \mathbb{R}_{\geq 0}}$ of $A$ are then monomial in the regular system of parameters $x_1, \dots, x_d$ (Proposition \ref{corollary-theorem-on-local-uniformization}). As $A$ has dimension $d$, its residue field $\kappa_A$ has transcendence degree $n - d$ over $k$. Then using the fact that $k$ is perfect, choose a separating transcendence basis $\{t_1, \dots, t_{n-d}\}$ of $\kappa_A/ k$, and pick $y_1, \dots, y_{n - d} \in A$ such that
$$y_i \equiv t_i \hspace{2mm} \text{mod} \hspace{1mm} \mathfrak m.$$
By \cite[VI, $\mathsection$10.3, Theorem 1]{Bou89}, $\{x_1, \dots, x_d, y_1, \dots, y_{n-d}\}$ is algebraically independent over $k$, and we obtain a local extension
$$j: k[x_1, \dots, x_d, y_1, \dots, y_{n-d}]_{(x_1, \dots, x_d)} \hookrightarrow A,$$
of local rings of the same dimension that is unramified by construction. Moreover, $j$ is also flat \cite[Theorem 23.1]{Mat89}, essentially of finite type, hence essentially \'etale. 

Let $\widetilde{A} \coloneqq k[x_1, \dots, x_d, y_1, \dots, y_{n-d}]_{(x_1, \dots, x_d)}$. It is easy to see that $\mathfrak{a}_{\bullet} \cap \widetilde{A} \coloneqq \{\mathfrak{a}_m \cap \widetilde{A}\}_{m \in \mathbb{R}_{\geq 0}}$ is the collection of valuation ideals of $\widetilde{A}$ with respect to the restriction of $\nu$ to $\text{Frac}(\widetilde{A})$. Moreover, if $S$ is a set of monomials in $x_1, \dots, x_d$ generating $\mathfrak{a}_m$, and $I_m$ is the ideal of $\widetilde{A}$ generated by $S$, then $I_m = I_mA \cap \widetilde{A} = \mathfrak{a}_m \cap \widetilde{A}$, where the first equality follows by faithful flatness of $j$. Thus, each $\mathfrak{a}_m \cap \widetilde{A}$ is generated by the same monomials in $x_1, \dots, x_d$ that generate $\mathfrak{a}_m$. Then to prove the theorem, it suffices to show by Corollary \ref{asymptotic-test-ideals-etale-maps} that
$$
\bigcap_{m \in \mathbb{R}_{\geq 0}} (\mathfrak{a}_m \cap \widetilde{A}: \tau_m(\mathfrak{a}_{\bullet} \cap \widetilde{A})) \neq (0).
$$

But now we are in the setting of Example \ref{test-ideals-monomials} since we are dealing with monomial ideals in the localization of a polynomial ring. We claim that 
$$ x_1\dots x_d \in \bigcap_{m \in \mathbb{R}_{\geq 0}} (\mathfrak{a}_m \cap \widetilde{A}: \tau_m(\mathfrak{a}_{\bullet} \cap \widetilde{A})).$$ 
Choose $\ell \in \mathbb{N}$ such that $\tau_m(\mathfrak{a}_{\bullet} \cap \widetilde{A}) = \tau((\mathfrak{a}_{\ell m} \cap \widetilde{A})^{1/\ell})$. Since $\mathfrak{a}_{\ell m} \cap \widetilde{A}$ is generated by 
$\big{\{}x^{a_1}_1\dots x^{a_d}_{d}: \sum a_i\nu(x_i) \geq \ell m\big{\}},$
and test ideals commute with localization, we then know by Example \ref{test-ideals-monomials} that $\tau_m(\mathfrak{a}_{\bullet} \cap \widetilde{A}) = \tau((\mathfrak{a}_{\ell m} \cap \widetilde{A})^{1/\ell})$ is generated by monomials $x^{b_1}_1 \dots x^{b_d}_d$ such that $(b_1 + 1, \dots, b_d + 1)$ is in the interior convex hull of
$$\bigg{\{}\bigg{(}\frac{a_1}{\ell}, \dots, \frac{a_d}{\ell} \bigg{)}: a_i \in \mathbb{N} \cup \{0\}, \sum \frac{a_i}{\ell}\nu(x_i) \geq m\bigg{\}}.$$
Then clearly $\sum (b_i + 1)\nu(x_i) \geq m$, that is, $(x_1\dots x_n) \cdot x^{b_1}_1\dots x^{b_d}_d \in \mathfrak{a}_m \cap \widetilde{A}$. This shows $(x_1\dots x_n) \cdot \tau_m(\mathfrak{a}_{\bullet} \cap \widetilde{A}) \subseteq \mathfrak{a}_m \cap \widetilde{A}$, for all $m \in \mathbb{R}_{\geq 0}$.
\end{proof}

\section{Consequences of Theorem \ref{Theorem-B}}
\label{Consequences of Theorem B}
Throughout this section $k$ is a perfect field of prime characteristic, $X$ a regular variety over $k$ with function field $K$, and $\nu$ a non-trivial, real-valued Abhyankar valuation of $K/k$ centered on $x \in X$.

\subsection{Proof of Theorem \ref{Theorem-A}} 
Our goal is to show that there exists $e \geq 0$ such that for all $m \in \mathbb R_{\geq 0}$, $\ell \in \mathbb N$,
$$\mathfrak{a}_m(X)^{\ell} \subseteq \mathfrak{a}_{\ell m}(X) \subseteq \mathfrak{a}_{m-e}(X)^{\ell}.$$
From now we also assume $m > 0$, as otherwise all the ideals equal $\mathcal{O}_X$. 

Let  $(\mathfrak{a}_{\bullet})_x \coloneqq \{\mathfrak{a}_m(\mathcal{O}_{X,x})\}_{m \in \mathbb{R}_{\geq 0}}$ denote the graded system of valuation ideals of the center $\mathcal{O}_{X,x}$, and using Theorem \ref{Theorem-B}, fix a nonzero $\tilde{s} \in \mathcal{O}_{X,x}$ such that
$$\tilde{s} \in \bigcap_{m \in \mathbb{R}_{\geq 0}} \big{(}\mathfrak{a}_m(\mathcal{O}_{X,x}): \tau_m((\mathfrak{a}_{\bullet})_x)\big{)}.$$
Define $e:= \nu(\tilde{s}).$

Since the inclusion $\mathfrak{a}_m^{\ell} \subseteq \mathfrak{a}_{\ell m}$ is clear, it suffices to show that for the above choice of $e$,
\begin{equation}
\label{star}
\Gamma(U, \mathfrak{a}_{\ell m}) \subseteq \Gamma(U, \mathfrak{a}_{m-e}^{\ell}),
\end{equation}
for all $m \in \mathbb R_{\geq 0}$, $\ell \in \mathbb N$, and affine open $U \subseteq X$. Furthermore, we may assume $U$ contains the center $x$ of $\nu$, as otherwise $\Gamma(U, \mathfrak{a}_{\ell m})$ and $\Gamma(U, \mathfrak{a}_{m-e}^{\ell})$ both equal $\mathcal{O}_X(U)$. We use $(\mathfrak{a}_{\bullet})_U$ to denote the collection $\{\mathfrak{a}_m(U)\}_{m \in \mathbb R_{\geq 0}}$ of valuation ideals of $\mathcal{O}_X(U)$.

Utilizing Lemma \ref{valuation-ideals-primary} and Lemma \ref{primary-ideals-Noetherian domain}(2), express $\tilde{s}$ as a fraction $s_U/t$, for some nonzero 
$$s_U \in \bigcap_{m \in \mathbb{R}_{\geq 0}} \big{(}\mathfrak{a}_m(U): \tau_m((\mathfrak{a}_{\bullet})_U)\big{)},$$ and $t \in \mathcal{O}_X(U)$ such that $t_x \in \mathcal{O}_{X,x}^{\times}$. Then $\nu(s_U) = \nu(\tilde{s}) = e$, and it follows that for all $m \in \mathbb{R}_{\geq 0}$,
$$\tau_m\big{(}(\mathfrak{a}_{\bullet})_U\big{)} \subseteq \mathfrak{a}_{m-e}(U).$$
Proposition \ref{some-properties-asymptotic-test-ideal}(2) implies that $\Gamma(U, \mathfrak{a}_{\ell m}) \subseteq \tau_m\big{(}(\mathfrak{a}_{\bullet})_U\big{)}^{\ell},$
and we obtain (\ref{star}) by observing that 
\[
\Gamma(U, \mathfrak{a}_{\ell m})  \subseteq \tau_m\big{(}(\mathfrak{a}_{\bullet})_U\big{)}^{\ell} \subseteq \mathfrak{a}_{m-e}(U)^{\ell} =  \Gamma(U, \mathfrak{a}_{m-e}^{\ell}). \qedhere
\]

\subsection{Proof of Corollary \ref{Corollary-C}}
We want to prove that if $\nu$, $w$ are two non-trivial real-valued Abhyankar valuations of $K/k$, centered on a regular local ring $(A, \mathfrak m)$ essentially of finite type over $k$ with fraction field $K$, then there exists $C > 0$ such that for all $x \in A$,
$$\nu(x) \leq Cw(x).$$
Our argument is similar to \cite{ELS03}, and is provided for completeness.

We let $\mathfrak{a}_{\bullet} = \{\mathfrak{a}_m\}_{m \in \mathbb{R}_{\geq 0}}$ denote the collection of valuation ideals of $A$ associated to $\nu$, and $\mathfrak{b}_{\bullet} = \{\mathfrak{b}_m\}_{m \in \mathbb{R}_{\geq 0}}$ the collection associated to $w$. Since $A$ is Noetherian, there exists a nonzero $x \in \mathfrak m$ such that for all nonzero $y$ in $\mathfrak m$,
$$w(x) \leq w(y).$$
Otherwise, one can find a sequence $(x_n)_{n \in \mathbb N} \subset \mathfrak m$ such that $w(x_1) > w(x_2) > w(x_3) > \dots,$ giving us a strictly ascending chain of ideals
$\mathfrak{b}_{w(x_1)} \subsetneq \mathfrak{b}_{w(x_2)} \subsetneq \mathfrak{b}_{w(x_3)} \subsetneq \dots \hspace{1mm}.$
For the rest of the proof, let
$$\delta \coloneqq \inf\{\nu(x): x \in \mathfrak{m} - \{0\}\}.$$

\begin{claim} 
\label{claim}
There exists $n > 0$ such that for all $\ell \in \mathbb{N}$, $\mathfrak{a}_{\ell n} \subseteq \mathfrak{b}_{\ell\delta}$. 
\end{claim}

Assuming the claim, let $C \coloneqq 2n/\delta$, and suppose there exists $x_0 \in \mathfrak m$ such that $\nu(x_0) > Cw(x_0).$
Now choose $\ell \in \mathbb N$ such that
\begin{equation}
\label{Izumi-in}
(\ell - 1)\delta \leq w(x_0) < \ell\delta. 
\end{equation} 
Such an $\ell$ exists by the Archimedean property of $\mathbb{R}$, and moreover, $\ell \geq 2$ since $w(x_0) \geq \delta$. Clearly, $x_0 \notin \mathfrak{b}_{\ell\delta}$, and multiplying (\ref{Izumi-in}) by $C$, we get
$$2(\ell-1)n \leq Cw(x_0) < 2\ell n.$$
But $\ell \geq 2$ implies $\ell n \leq 2(\ell - 1)n \leq Cw(x_0) < \nu(x_0)$. Then $x_0 \in \mathfrak{a}_{\ell n}$, contradicting $\mathfrak{a}_{\ell n} \subseteq \mathfrak{b}_{\ell\delta}$. 

\begin{proof}[\textbf{Proof of Claim \ref{claim}:}]
By our choice of $\delta$, $\mathfrak{b}_{\delta} = \mathfrak m$. 
Thus, for all $\ell \in \mathbb N$, $\mathfrak{m}^{\ell} \subseteq \mathfrak{b}_{\ell\delta}.$
Since by Theorem \ref{Theorem-B} 
\begin{equation}
\label{something}
\bigcap_{m \in \mathbb{R}_{\geq 0}}\big{(}\mathfrak{a}_m: \tau_m(\mathfrak{a}_{\bullet})\big{)} \neq (0),
\end{equation}
there must exist some $n > 0$ such that $\tau_n(\mathfrak{a}_{\bullet}) \subseteq \mathfrak m.$ Otherwise, for all $m \in \mathbb R_{\geq 0}$, $\tau_m(\mathfrak{a}_{\bullet}) = A$, which would imply that any $s \in \bigcap_{m \in \mathbb{R}}\big{(}\mathfrak{a}_m: \tau_m(\mathfrak{a}_{\bullet})\big{)}$ is also an element of
$\bigcap_{m \in \mathbb{R}_{\geq 0}} \mathfrak{a}_m = (0),$
contradicting (\ref{something}). Then by Proposition \ref{some-properties-asymptotic-test-ideal}(2), for all $\ell \in \mathbb N$,
\[\mathfrak{a}_{\ell n} \subseteq \tau_{n}(\mathfrak{a}_{\bullet})^{\ell} \subseteq \mathfrak m^{\ell} \subseteq \mathfrak{b}_{\ell\delta}. \qedhere \]
\end{proof}

\subsection{Examples} We now provide some examples to show that Theorem \ref{Theorem-A}, Theorem \ref{Theorem-B} and Corollary \ref{Corollary-C} are false if the valuation $\nu$ is not Abhyankar. These examples have been taken from the author's dissertation \cite{Dat18}. 

\begin{examples}
{\*}
\begin{enumerate}
\item Uniform approximation of valuation ideals (Theorem \ref{Theorem-A}) \emph{fails} in general for real-valued valuations that are not Abhyankar. The discrete valuation $\nu_{q(t)}$ of $\mathbb{F}_p(X,Y)$ constructed in Example \ref{examples-Abhyankar-valuations}(4) provides a counterexample to Theorem \ref{Theorem-A}. Recall that $\nu_{q(t)}$ is obtained as the composition
$$\mathbb{F}_p(X,Y)^{\times} \hookrightarrow \mathbb{F}_p((t))^{\times} \xrightarrow{t-adic} \mathbb{Z},$$
by mapping $X \mapsto t$ and $Y \mapsto q(t)$ such that $t, q(t)$ are algebraically independent over $\mathbb{F}_p$. We can choose
$$q(t) = a_1t + a_2t^2 + a_3t^3 + \dots,$$
such that $a_1 \neq 0$. Then $\nu_{q(t)}$ is centered on 
$$A \coloneqq \mathbb{F}_p[X,Y]_{(X,Y)}.$$ 
Now for any $m \in \mathbb{N}$, the valuation ideal $\mathfrak{a}_m$ of the center $A$ contains the ideal 
$$(X^m, Y - a_1X + a_2X^2 + \dots + a_{m-1}X^{m-1}).$$
Therefore $A/\mathfrak{a}_m$ has length $\leq m$. 

Suppose there exists $e$ as in Theorem \ref{Theorem-A}. Fixing $m \in \mathbb{N}$ such that 
$$m > e,$$ 
we see that for all $\ell \in \mathbb{N}$, the length of $A/\mathfrak{a}_{\ell m}$ is $\leq \ell m$. In other words, for a fixed $m$, the length of $A/\mathfrak{a}_{\ell m}$ grows as a linear function in $\ell$. On the other hand,
$$\mathfrak{a}^{\ell}_{m - e} \subseteq (X,Y)^{\ell}.$$
Thus the length of $A/\mathfrak{a}^{\ell}_{m-e}$ is at least the length of $A/(X,Y)^{\ell}$, and the latter grows as a quadratic function in $\ell$. Hence $\mathfrak{a}_{\ell m}$ cannot possibly be contained in $\mathfrak{a}^{\ell}_{m-e}$ when $\ell \gg 0$, thereby providing a counterexample to Theorem \ref{Theorem-A}. Since Theorem \ref{Theorem-A} is a formal consequence of Theorem \ref{Theorem-B}, it follows that Theorem \ref{Theorem-B} must also be false for real-valued valuations that are not Abhyankar. 

\item Izumi's theorem (Corollary \ref{Corollary-C}) also fails in general when the valuations $\nu$ and $w$ are not both Abhyankar. To see this, we take one valuation to be the unique valuation $\nu_\pi$ on $\mathbb{F}_p(X,Y)$ such that
$$\nu_\pi(X) = 1 \hspace{1mm} \textrm{ and} \hspace{3mm} \nu_\pi(Y) = \pi.$$
Note $\nu_\pi$ is an Abhyankar valuation of $\mathbb{F}_p(X,Y)/\mathbb{F}_p$ since
$$\td{\mathbb{F}_p(X,Y)/\mathbb{F}_p} = 2 = \dim_{\mathbb Q}(\mathbb{Q} \otimes_{\mathbb Z} \Gamma_{\nu_\pi}).$$
Choose the other valuation to be of the form $\nu_{q(t)}$, where, specifically,
$$q(t) = \sum_{i = 1}^{\infty} t^{i!}.$$
It is not difficult to check that $t, q(t)$ are algebraically independent over $\mathbb{F}_p$ (see also \cite[Chapter VI, $\mathsection 3$, Exercise 1]{Bou89}), so that the valuation $\nu_{q(t)}$ is indeed well-defined. 

Both $\nu_\pi$ and $\nu_{q(t)}$ are centered on $\mathbb{F}_p[X,Y]_{(X,Y)}$. For all $n \in \mathbb{N}$, defining
$$x_n \coloneqq Y - \sum_{i = 1}^n X^{i!},$$
we see that,
$$\nu_{\pi}(x_n) = 1 \hspace{1mm} \textrm{ and} \hspace{3mm} \nu_{q(t)}(x_n) = (n+1)!.$$ 
Clearly there does not exist a fixed real number $C > 0$ such that for all $n \in \mathbb{N}$,
$$\nu_{q(t)}(x_n) = (n+1)! \leq C = C\nu_{\pi}(x_n).$$
Thus, Izumi's theorem fails when the real-valued valuations are not Abhyankar. 
\end{enumerate}
\end{examples}

\section{Acknowledgments}
I thank my advisor, Karen Smith, for her suggestion to pursue this line of research, for her support, and for always generously sharing her ideas. I benefitted from discussions with Karl and Linquan Ma in Utah and later at the 2017 Arizona Winter School. I thank Manuel Blickle, Mel Hochster, Johan de Jong, Hagen Knaf, Tiankai Liu, Takumi Murayama, Axel St\"abler for useful conversations, and Harold Blum for explaining his proof of Theorem \ref{Theorem-A} in characteristic $0$ using log-discrepancies of valuations \cite{Blu18}. In addition, Harold and Takumi's careful reading of a draft has improved this paper's clarity. Finally, I thank the referee for helpful suggestions. 

\bibliographystyle{amsalpha}

\end{document}